\documentclass[12pt]{article}

\usepackage{xcolor}
\usepackage{amssymb}
\usepackage{mathrsfs}

\usepackage{common}

\title{Stable amalgamation over a predicate and the Gaifman property}

\author{Saharon Shelah\footnote{Publication no. 322b on Shelah's list of publications. Shelah thanks the Israel Science Foundation
   (Grants 1053/11, 1838/19), and the
   European Research Council (Grant 338821)
   for partial support of this research. Shelah is grateful to Craig Falls for funding typing services that were used during the work on the paper.}  and  Alexander Usvyatsov\footnote{Research partially supported by Marie Sklodowska Curie CIG 321915 "ModStabBan". Usvyatsov thanks the Austrian Science Foundation (FWF), projects P33895 and  P33420 for their partial support at different stages of this research.
    }}

\newcommand{\Addresses}{{
  \bigskip
  \footnotesize

 Saharon Shelah, \textsc{Mathematics Department,
Hebrew University of Jerusalem,
91904 Givat Ram, Israel}
 
 \medskip

Alexander Usvyatsov, \textsc{Institut f\"{u}r Diskrete Mathematik und Geometrie,
TU Wien,
1040 Vienna, Austria}

\medskip

}}

\newtheorem{theorem}{Theorem}[section]
\newtheorem{definition}[theorem]{Definition}
\newtheorem{example}[theorem]{Example}
\newtheorem{lem}[theorem]{Lemma}
\newtheorem{obs}[theorem]{Observation}
\newtheorem{co}[theorem]{Corollary}

\newtheorem{hyp}[theorem]{Hypothesis}
\newtheorem{remark}[theorem]{Remark}

\newtheorem{no}[theorem]{Notation}
\newtheorem{con}[theorem]{Conjecture}

\newtheorem{ft}[theorem]{Fact}
\renewenvironment{proof}{\noindent {\em Proof:}}{\hspace*{1cm}
        \hspace*{\fill}$\rule{1.2ex}{1.4ex}$\medskip} 
\newenvironment{re}{\begin{remark}\rm}{\end{remark}}
 
\newenvironment{de}{\begin{definition}\rm}{\end{definition}}
\newenvironment{fact}{\begin{ft}\rm}{\end{ft}}

\newcommand{\red}[1]{{\color{black}{#1}}}

\newcommand{\blue}[1]{{\color{black}{#1}}}
\newcommand{\gray}[1]{{\color{black}{#1}}}

\begin{document}
\baselineskip 24pt

\date{}

\maketitle

\abstract{We consider the following property of a first order theory T with a distinguished unary predicate P: every model of the theory of P occurs as the P-part of some model of T. We call this property the Gaifman property. Gaifman conjectured that if T is relatively categorical over P, then it has the Gaifman property. We propose a generalized version of this conjecture: if  T fails the Gaifman property, then it exhibits non-structure over P, i.e., has many non-isomorphic models over P in many cardinalities. We address this conjecture for countable theories. Motivated by ideas from  Classification Theory, we separate this conjecture into two parts:  1) stability over P (a structure property of theories) implies the Gaifman property, and 2) instability over P implies non-structure. In this paper prove the first part of this conjecture. In fact, we prove a stronger statement: an appropriate version of stability implies higher stable amalgamation properties.}

\section{Introduction}


In \cite{gaifman}, Gaifman has conjectured that, if a countable theory $T$ is categorical over a unary predicate $P$, then every model of the theory of $P$ ``occurs'' as the $P$-part of some model of $T$. We will refer to the latter property of theories
as the \emph{Gaifman property} over $P$; that is, we say that $T$ has the Gaifman  property over $P$ (or just \emph{$T$ has the Gaifman property}, when $P$ is clear from the context) when every model of the theory of $P$ is the $P$-part of some model of $T$. Gaifman \cite{Ga} proved this conjecture in case $T$ is \emph{rigid} over $P$ (see also \cite{benedikt-pred} for a recent ``effective'' version of this result). Hodges et al (e.g., \cite{Hod-cat1, Hod-cat2, Hod-cat3}) investigated Gaifman's conjecture and related questions for specific types of theories (an abelian group with a predicate picking out a subgroup, a pair of linear orders). 

The first author has proven an ``absolute'' version of this conjecture in \cite{Sh234}. Specifically, it follows from results there that if $T$ is \emph{absolutely categorical} over $P$ (the categoricity is preserved in forcing extensions), then $T$ has the Gaifman property over $P$. As far as we know, the general statement is still open. 

In this paper we formulate and address a much stronger version of Gaifman's Conjecture (which we refer to as the ``Generalized Gaifman's Conjecture''):

\begin{con}\label{con:gaifman1}
	Let $T$ be a countable complete theory, $P$ a distinguished unary predicate in its vocabulary. Assume that $T$ fails the Gaifman property. That is, assume that there exists a model $N$ of the theory of $P$ such that for no $M \models T$ is it the case that $P^M = N$. Then for every regular cardinal $\lam$ big enough, and every $\mu \ge \lam$, $T$ has $2^\lam$ models of cardinality $\mu$, which are non-isomorphic over $P$.  
\end{con}

This conjecture states that, if a countable theory $T$ fails the Gaifman property, then it exhibits ``non-structure over $P$''. In other words, we believe that a much weaker version of ``structure'' for $T$ over $P$ (than categoricity) is enough for the Gaifman property. 

\smallskip

In the present paper, we break the above conjecture into two parts, and prove one of them. The motivation for our approach (and indeed for the above statement of the conjecture) comes from \emph{Classification theory}, a study of dichotomies in model theory, developed by the first author (e.g., \cite{Sh:c}). Broadly speaking, classification theory is a programme of search for \emph{dividing lines}. A dividing line is a property of first order theories (or, more generally, of classes of models), such that theories that have this property (the ones falling on the ``structure'' side of the dividing line) have various ``positive structure'' properties, whereas theories that fail the property in questions (fall on the ``non-structure'' side of the dividing line) exhibit ``non-structure''. The intuition is such dividing lines separate, in some sense, between ``classifiable'' theories (whose models can be described using a ``reasonable'' collection of invariants, e.g., dimensions), and ``non-classifiable'' ones, where such classification is impossible. 

One of the most fundamental and best studied dividing lines is \emph{stability}, introduced (first in classical model theory) by the first author, generalizing the work and ideas of Morley \cite{Mor}. If (a countable theory) $T$ is unstable, it has the maximal number of non-isomorphic models (and even $\aleph_1$-saturated models) in any uncountable cardinality. On the other hand, if $T$ is stable, it has a very well-behaved notion of independence, that leads to good notions of dimension, and many other positive structural properties. These classical results appear in \cite{Sh:c}.

It is our hypothesis that the appropriate notion of stability also plays a central role in the study of model theory over a predicate. However, we believe that, unlike in the classical (first order) context, stability over a predicate gives rise to more than one (indeed, infinitely many) dividing lines. We will make this idea of  ``$\omega$ levels of classification (introduced by the first author in \cite{Sh234}) of  ``$n$-stability over $P$'' for $n<\om$ more precise in the next section.

One goal of this work is to investigate \emph{structure} consequences for theories that fall on the positive side of all these dividing lines. In particular, we  prove the following:

\begin{theorem}
 \label{thm:main-intro}
 Let $T$ be a countable complete theory, $P$ a distinguished unary predicate in its vocabulary, such that $P$ is ``very stably embedded'' (by which we mean that $P$ is stably embedded, and every subset of $P$ definable in $T$, is definable already in $T^P$, the theory of $P$). Assume that $T$ is ``$n$-stable'' for all $n<\om$. Then $T$ has the Gaifman property. 
\end{theorem}

The next section we discuss the  setting of this paper, the notion of $n$-stability mentioned in the statement of the Theorem above, our general approach, and the main technical tools.  At the end, we will be able to state a stronger version of Theorem \ref{thm:main-intro} that is actually proven in this paper, and explain the connections with previous work. 

\section{Background: the existence property and higher amalgamation}

\subsection{Background and setting}

From now on, $T$ will be a (fixed) complete first order theory, and $P$ a distinguished unary predicate in the vocabulary of $T$. For simplicity, we assume that the vocabulary of $T$ has no function symbols, and that $T$ implies that $P$ is infinite.

Let $\mathcal{C} $ be the monster model of $T$. From now on, we assume that all models of $T$ are elementary submodels of $\mathcal{C}$, and all sets are subsets of $\mathcal{C}$. \medskip

For a set $A$, we denote $P^A = A \cap P^\cC$. We also denote by $\cC^P$ the substructure of $\cC$, which (as a set) equals $P^\cC$, and we let $T^P$ be the theory of $\cC^P$. 

Throughout the paper, we make the following ``structure'' assumptions:


\begin{hyp}
 \label{hyp:1} (\underline{Hypothesis 2}).
 
\begin{enumerate}
 \item $P$ is stably embedded.
 \item Every definable subset of $P^\cC$ is already definable in $T^P$. 
\end{enumerate}
\end{hyp}


We may assume these without loss of generality while investigating theories that do not exhibit ``non-structure''. See a detailed discussion on this in \cite{ShUs322a}. The proofs of the relevant non-structure results appear in \cite{PiSh130}. Because of clause (ii) of Hypothesis \ref{hyp:1} above, we may also assume that $T$ has been Morley-ized (hence has quantifier elimination, even down to the level of predicates). Again, see \cite{PiSh130} or \cite{ShUs322a} for a more detailed discussion.  

\smallskip 

Clearly, $\cC^P$ can be viewed as the monster model of $T^P$. 

When no confusion should arise, we will write $P$ for $P^\cC$. Also, for a set $A$, we will often denote by $A$ both the set and the substructure of $\cC$ with universe $A$. So for example, when we write that $A\cap P^\cC$ is $\lam$-saturated, or just that $A\cap P$ is $\lam$-saturated, we mean that the substructure of $\cC$ with universe $P^A$ is a $\lam$-saturated model of the appropriate theory (if $A \cap P \prec P$, which will be the case in this paper, then the appropriate theory is $T^P$). 

For two sets $A, A'$, we will write $A \equiv A'$ if $A$ and $A'$ are (universes of) elementarily equivalent substructures of $\cC$. 

\subsection{Basic definitions}

We recall some basic concepts from \cite{PiSh130}, \cite{ShUs322a}, and \cite{Us24}. 

\begin{de}\label{de:existence}
\begin{enumerate}
\item 
	We say that a set $A$ has the \emph{existence property over $P$}, or simply the \emph{existence property} if there exists $M \models T$ such that $A \subseteq M$ and $P^M = P^A$. 
\item
	We say that $T$ has the \emph{Gaifman property} if every $N \models T^P$ has the existence property. 
\end{enumerate}
\end{de}

In order for a set $A$ to have the existence property, $P^A$ should be ``suitable'' for being the $P$-part of a model of $T$. For example, $P^A$ should obviously be itself a model of $T^P$; moreover, $P^A$ has to be ``closed enough''. 


This closedness condition is summarized in the following definition.


\begin{de}\label{6}\label{dfn:complete}
$A\subseteq {\cal C}$ is called {\it complete} if for every
formula $\psi(\bar x,\bar y)$ and $\bar b\subseteq A, \models
(\exists\bar x\in P)\psi(\bar x,\bar b)$ implies $(\exists
\bar a\subseteq P\cap A)\models \psi(\bar a,\bar b)$. 
\end{de}

%
%
%


The following is clear:

\begin{obs}\label{obs:complete}
If $M\prec {\cal C}$ and $P^M\subseteq A\subseteq M$, then
$A$ is complete.
\end{obs}


This confirms that in order for a set $A$ to have the existence property, it has to be complete. In this paper we are interested in the converse.

The following useful characterization offers another understanding of the notion of completeness (see Observation 4.2 in \cite{ShUs322a}):

\begin{fact}
\label{obs:complete_characterization}
A set $A$ is complete if and only if for every $\bar a\subseteq A$ and
$\psi(\bar x,\bar y)$  the $\psi$-type $tp_\psi (\bar a/P^{\cal C})$ is definable
over $A\cap P^{\cal C}$ and $A\cap P^{\cal C}\prec P^{\cal C}$.
\end{fact}

Now we recall the definitions of relevant types and stability. Note that these notions are only defined over complete sets.

\begin{de}\label{6.5}\label{dfn:startypes}

Let $A$ be a complete set. 

\begin{itemize}

\item[(i)]  Denote $$S_*(A)=\{tp(\bar c/A):P\cap (A\cup \bar c)=P\cap A {\rm\ and}\ A\cup \bar
c\ {\rm is\ complete}\}$$
\item[(ii)] $A$ is \emph{stable over $P$}, or simply \emph{stable}, if for all $A^\prime$
with $A^\prime\equiv A$, we have $|S_*(A^\prime)|\leq |A^\prime|^{|T|}$.
\end{itemize}
\end{de}

\begin{remark}
\begin{enumerate}
\item Even though ``stability over $P$'' is a more appropriate and accurate name for our notion of stability of a set (and the term ``stable set'' exists in literature, and has a different meaning), since we  have only one notion of stability in this article (stability over $P$), we will sometimes omit ``over P'' and simply write ``stable''. 
\item We will refer to types in  $S_*(A)$ as \emph{complete types over $A$ which are weakly orthogonal to $P$}. 
\end{enumerate}

\end{remark}

\subsection{Finite independent systems of models and higher amalgamation properties}

The main technical tool in this paper is so-called ``good independent systems of models'', or just \emph{good systems}.  Informally, such a system corresponds to a collection of models of $T$ and $T^P$ that are ``independent from each other''.  A formal definition (which is somewhat long and quite technical) appears in section 5. The main idea is that, given a partial ordered set $I$, an $I$-system is a collection of sets $A_u$ for $u \in I$, satisfying a collection of conditions (some very natural, some a bit more technical). In our case, we will always assume that $I \subseteq \cP(n)$ for some $n$, where $\cP(n) = \cP(\set{0, \ldots  n-1}) = \set{u \colon u \subseteq \set{0, \ldots, n-1}}$. This will allow us to distinguish between elements of the system that are models of $T$ and those that are models of $T^P$. Specifically, we will (quite arbitrarily, but this turns out to be a convenient choice) assume that if $0 \notin u$, then $A_u \models T^P$, otherwise $A_u \models T$. Moreover, we assume that $P^{A_u} = A_{u \setminus\{0\}}$. 

We will be most interested in good $\cP(n)$-systems (i.e., $I = \cP(n)$) and $\cP^{-}(n)$-systems, where $\cP^{-}(n)=\cP(n)\setminus \set{n}$. For simplicity of notation, in this section, we shall refer to the former as an $n$-system, and the latter as an  $n^-$-system (we avoid the use of this shortcut later in the paper). 

Note that:

\begin{itemize}
\item
A $1$-system consists of $A_{\emptyset}$ and $A_{\set{0}}$, where $A_{\set{0}} \models T$ and $A_\emptyset = P^{A_{\set{0}}}$ 
\item
A $1^-$-system consists only of $A_\emptyset \models T^P$ 
\item
A $2^-$-system consists of $A_{\emptyset}, A_{\set{1}} \models T^P$ and $A_{\set{0}}$ with $A_\emptyset = P^{A_{\set{0}}}$ 

\end{itemize}

Additional requirements on a $2^-$-system demand, for example, that $A_\emptyset \prec A_{\{1\}}$, and that $A_{\{0\}}$ is ``independent from $A_{\{1\}}$ over $A_\emptyset$ (in this specific case this just means non-forking independence: for every $\a \in A_{\{0\}}$, the type $\tp(\a/A_{\{1\}})$ is definable over $A_\emptyset$, and is ``automatic'', since $P$ is stably embedded, and $A_{\{0\}}$ is a model, hence a complete set). in a $2$-system, one also has a model $A_2 = A_{\set{0,1}} \models T$, which contains all the above sets, and, indeed, $A_{\set{0}} \prec A_2$, and $P^{A_2} = A_{\set{1}}$.

In the case of $n=3$, the requirements are already more technical, and we will not discuss them here (see section 5 for a general definition). 

Sometimes one wants to discuss a particular type of system, where, for example, the model $A_n$ is always atomic, constructible, or saturated over the ``smaller'' sets. For example, in \cite{Sh234}, the first author studied ``u.l.a.'' systems, i.e., systems in which $A_n$ is \emph{uniformly locally atomic} (u.l.a.) over the union of the appropriate $n^-$-system (see subsection \ref{sub:previous}). 

The property of $n$-stability mentioned in Theorem \ref{thm:main-intro}, corresponds to the fact that every $n^-$-system (more precisely, the union of every $n^{-}$-system) is a complete set stable over $P$. We say that $T$ has $n$-existence if every $n^-$-system can be ``completed'' to an $n$-system. What we really prove in the Main Theorem in section 6 is that $n$-stability for all $n$ implies $n$-existence for all $n$. Moreover, the model $A_n$ can, in this case, be taken to be locally constructible over the union of the $n^-$-system. 

In particular, we obtain the following:

\begin{theorem}
 \label{thm:main-intro2}
 Let $T$ be a countable complete theory, $P$ a distinguished unary predicate in its vocabulary, such that $P$ is ``very stably embedded'' (as in Theorem \ref{thm:main-intro})). Assume that $T$ is ``$n$-stable'' over $P$ for all $n<\om$. Then the union of every $n^-$-system has the existence property. 
\end{theorem}

So Theorem \ref{thm:main-intro}, as stated in the introduction, is a particular case of the Theorem \ref{thm:main-intro2} where (in the conclusion) $n$ is taken to be equal to $1$. 


The property of $n$-existence can also be interpreted as an appropriate  ``stable $n$-amalgamation'' property: every $n^{-}$-system of models can be amalgamated (i.e., embedded into one big model) without increasing the $P$-part. 

\subsection{Why systems?}

One could ask: why bother with $n$-systems, if all one is interested is the Gaifman property (the case $n=1$)? There are several potential answers to this question.

 First, we do not know a general proof of the Gaifman property that does not go through a proof of a stronger result (the existence property for a larger class of complete sets). 
 
 Second, $n$-systems are interesting objects in their own right, and $n$-existence, which (as we have mentioned above) in a way corresponds to a higher amalgamation property, can be quite useful, both in theoretical model theoretic arguments, and in applications to particular theories (in the same way as \emph{excellence} in abstract elementary classes can be used in particular examples to obtain interesting consequence, as is done in e.g. \cite{Kirby2013}). Problems related to  ``higher amalgamation'' ($n$-existence) have generated much interesting research, with connections to e.g., algebraic topology (\cite{hrushovski2024groupoidsimaginariesinternalcovers, Goodrick_Kolesnikov_2010, goodrick2011amalgamationfunctorshomologygroups, Goodrick_Kim_Kolesnikov_2013}, to only name a few).

Third, we believe that $n$-stability is a ``true'' dividing line for all $n<\om$ in the following sense: we believe that one should be able to prove, assuming that $T$ is $n$-stable (over $P$), but $n+1$-unstable, that $T$ exhibits non-structure over $P$ (and this will hopefully be addressed in subsequent work). 

Finally, systems are a convenient and useful tool. Specifically, working with $n$-systems allows proofs by induction on $\lambda$, the cardinality of the system. 

We would like to (informally) illustrate this on the following ``baby'' example. 

Let's say we are interested in the Gaifman property (which is essentially equivalent to the problem of $1$-existence). That is, given a model $N \models T^P$, we ask whether there is $M\models T$ with $P^M = N$. In the language of systems, given a $1^-$-system, we ask whether it can be extended to a $1$-system. 

Let's say that we somehow already know that $T$ has $1$-existence and $2$-existence in cardinality $\lam$; and let us  deduce by an inductive argument that it has $1$-existence in $\lambda^+$.  Given a model $N \models T^P$ of cardinality $\lam^+$, we write it as an increasing union $N_i$ of models of $T^P$ of cardinality $\lam$. By $1$-existence in $\lambda$, there is $M_0 \models T$ of cardinality $\lambda$ whose $P$-part is $N_0$. By 2-existence in $\lambda$, there is $M_1 \models T$ of cardinality $\lambda$, containing $M_0$ and $N_1$, whose $P$-part is $N_1$. Continuing by induction, we obtain an increasing chain $M_i \models T$ whose $P$-part is $N_i$, and their union $M$ will be the desired model of $T$ of cardinality $\lambda^+$, whose $P$-part is $N$. 

In fact, this is a general phenomenon that makes systems so incredibly useful: they allow inductive transfer of properties between  cardinalities. It is very often the case that a property for $n+1$-systems in $\lambda$ implies the same property for $n$-systems in $\lambda^+$. This exactly is the case for $n$-existence. 
The proof of the Main Theorem (Theorem \ref{thm:main-intro2}) uses a version of an inductive argument described above (for all $n<\om$ at once), and the definition of a system, that may appear quite technical at first (and second) glance, is designed to make arguments of this type work.

\subsection{Previous work}\label{sub:previous}

In \cite{Sh234}, the first author has established a similar result under the assumption that $T$ has ``absolutely no two-cardinal models'' (a condition that the second author has called \emph{nulldimensionality} in \cite{Us24}, since it is equivalent to: every type over a model of $T$ which is ``weakly orthogonal'' to $P$ is algebraic; i.e., there are no ``free'' dimensions that are orthogonal to $P$). This is a very strong hypothesis, that does not hold in most of the interesting examples, such as vector spaces over a field, $ACFA_0$, $ECF$, etc (nulldimensionality is a natural assumption, however, if one is interested in investigating theories that are \emph{categorical over $P$}; in particular, if one's goal is to address the original Gaifman's Conjecture).  Under these circumstances, the first author could work with  a special kind of independent systems of models, in which every model $A_n$ is uniformly locally atomic (u.l.a.) over the union $\bigcup_{u \subset n} A_u$. In the general case (where $T$ is not nulldimensional), this notion of a system is too strong. In section 5 below, we develop a  theory of \emph{general} systems, that we use in order to prove an analogous result without the assumption of nulldimensionality.

In \cite{Us24}, the second author proves the \emph{full existence} property (all complete sets have the existence property) assuming \emph{full stability over $P$} (all complete sets are stable).It is also observed there that some natural examples (such as $ACFA_0$) are fully stable. However, it does not seem likely that this result brings us closer to a generalized version of Gaifman's Conjecture, since we do not know how to approach the ``non-structure'' side of full stability, or, indeed, whether the negation of full stability implies any kind of non-structure. On the other hand, we believe that techniques similar to those used by the first author in \cite{Sh234} in order to establish non-structure assuming instability of u.l.a. $n$-systems, can be strengthened and generalized to the more general notion of a system defined here. We therefore believe that the results in this paper are indeed the first step towards the proof of Conjecture \ref{con:gaifman1}. 

Finally, many particular cases of the  Gaifman property (and, more generally, the existence property) have been proven by numerous authors. Let us just give a few examples. Hodges (partly in collaboration with Yakovlev) \cite{Hod-cat1,Hod-cat2,Hod-cat3} has established the Gaifman property for certain classes of abelian groups (with a predicate for a distinguished subgroup) and pairs of linear orders. It follows from the work of Lachlan that the (full) existence property is always true for a countable stable theory $T$ (the second author's work \cite{Us24} mentioned 	above generalizes this result). Afshordel \cite{MR3343525} proved that every pseudo-finite field occurs as the fixed field of some difference closed field. Hence (any completion of) the theory $ACFA$ (algebraically closed fields with a generic automorphism $\sigma$) has the Gaifman property, where $P$ is interpreted as the fixed field of $\sigma$. Kirby and Zilber \cite{KiZil} proved the existence property for the class of exponentially closed fields (where $P$ is the kernel of the exponent function), assuming that $P$ is $\aleph_0$-saturated.

\section{Basic facts}
It is our intention to make this paper relatively self content and accessible to any expert in model theory. We will therefore begin by recalling some basic properties of the notions defined in the previous section (mostly from \cite{ShUs322a} and \cite{Us24}).

First, we restate the assumption that $P$ is stably embedded in a more uniform way (this is standard, but see Observations 4.1 in \cite{Us24}).

\begin{obs}\label{obs:uniformdef}
%
 There are $\langle \Psi_\psi(\y,\z) :\psi(\bar
x,\bar y)\in L(T)\rangle $ such that for all
$\bar a\subseteq \cC$, $tp_\psi (\bar a/P^\cC)$ is definable by
$\Psi _\psi (\bar y,\bar c)$ for some $\bar c\subseteq P^\cC$. 

In other words, for every $\psi(\x,\y)$ and $\a \in \fC$ there exists $\c \in P^\cC$ such that $\cC \models \forall \y \psi(\a,\y) \longleftrightarrow \Psi(\y,\c)$.

\end{obs}

Now we recall a few basic consequences of completeness. 

\begin{obs} \label{obs:defcomplete}(Observation 3.2 in \cite{Us24}) For any complete $A$ and
for all
$\bar a\subseteq A$, $tp_\psi (\bar a/P\cap A)$ is definable by
$\Psi _\psi (\bar y,\bar c)$ for some $\bar c\subseteq A\cap P$
(where $\Psi _\psi (\bar y,\bar z)$ is as in Observation \ref{obs:uniformdef}).
\end{obs}

\begin{fact}\label{fct:typeoverP}(Corollary 4.7 in \cite{ShUs322a})
	Let $A$ be a complete set, $p(x)$ a (partial) type over $A$ with $P(x) \subseteq p$. Then $p$ is equivalent 
	to a $T^P$-type $p'$ over $P^A$ with $|p'|=|p|$. 
	
	In particular, if $p$ is finite, then it realized in $P^A$. Similarly, if $|p|<\lam$ and $P^A$ is $\lam$-compact.
\end{fact} 

\begin{lem}(Lemma 4.12 in  \cite{ShUs322a}) (\underline{The Small Type Extension Lemma})\label{extension}\label{9}\label{le:typeextension} If $A \prec \cal C$ is saturated (or just $A\cap P$ is $|A|$-compact) and $p(\bar x)$ is an $L(T)$-type over $A$
of cardinality $<|A|$,  then there is some $\red{p^*}(\bar x)\in S_*(A)$
extending $p$.
\end{lem}

\begin{co}\label{approx}\label{8}\label{co:complete_exnetdabletomodel}(Proposition 4.13 in \cite{ShUs322a}).
Suppose that $A\cap P$ is $|A|$-compact. Then $A$ is complete if and only if $A$ has the existence property (i.e., there exists an $M\prec {\cal C}$ with $P^M\subseteq A\subseteq M$).
If $|A|=|A|^{<|A|}>|T|$, we can add ``$M$ saturated''.
\end{co}

\begin{fact}\label{fct:complete_char}(see Theorem 4.7 in \cite{Us24})
Let $A$ be a set. Then the following are equivalent:
\begin{enumerate}
\item $A$ is complete
%

%

\item
$P^A \elem P^\cC$, and:

For every finite type $p_0(\x)$ over $A$ (not necessarily realized in $A$) and a formula $\psi(\x,\b,\y)$  over $A$ , there is  $\d \subseteq P^A$ such that  the following is a (finite) type over $A$:
\[
p_0(\x) \cup \left\{  (\forall \y \subseteq P) \left[ \psi(\x,\b,\y) \longleftrightarrow \Psi_{\psi}(\y,\d) \right]  \colon 
	i<k \right\}
%
\]
where  $\Psi_\psi(\y,\z)$ is the defining formula for $\psi = \psi(\x\x',\y)$.

\item For every finite type $p_0(\x)$ over $A$ (not necessarily realized in $A$) and every formula $\psi(\x,\b,\z)$ over $A$, there is  $\d \subseteq P^A$ such that the following is a (finite) type over $A$:

\[
	p_0(\x) \cup \left\{  \left[(\exists \z \subseteq P)  \psi(\x,\b,\z)\right] \longrightarrow \psi(\x,\b,\d)   \right\}
\]

\end{enumerate}

\end{fact}

%


Finally, we recall that, although stability over $P$ is not a property of a specific set $A$, but rather of its theory as a substructure, it is enough to consider one saturated $A' \equiv A$:

\begin{fact} (Corollary 5.5 in \cite{ShUs322a})\label{fct:satstable}
	In Definition \ref{6}(iv), it is not necessary to consider all $A' \equiv A$. 
	More specifically, a complete set $A$ is stable if and only if $|S_*(A')| \le |A'|^{|T|}$ for some $A' \equiv A$ saturated, $|A'|>|T|$. 
	
	Moreover, it is enough that for $A'$ as above, $|S_*(A')| < 2^{|A'|}$.
\end{fact}

\section{Local Constructibility}

\subsection{Isolation and atomicity}

Let us recall the definitions of the notions of isolation relevant for the discussion in this article. 

\begin{de}
\begin{enumerate}
\item A (partial) type $p$ over a set $A$ is called \emph{locally isolated} (l.i.) if for every formula $\ph(x,y)$ there exists a formula $\theta_\ph(x,a_\ph) \in p$ such that $\theta_\ph(x,a_\ph) \vdash p\rest\ph$. We say that $p$ is locally isolated (l.i.) over $B \subseteq A$ if for every $\ph(x,y)$, its isolating formula $\theta_\ph$ is over $B$ (so $a_\ph \in B$). 

\item A (partial) type $p$ over a set $A$ is called \emph{$\lam$-isolated} if there exists a subset $r \subseteq p$ with $|r|<\lam$ such that $r \equiv p$. We say that $p$ is $\lam$-isolated over $B \subseteq A$ if $r$ as above is a partial type over $B$.

\end{enumerate}
\end{de}

\begin{re}
\begin{enumerate}
\item So $p$ is locally isolated if for every formula $\ph$, the restriction of $p$ to a $\ph$-type is implied by a single formula in $p$ (which does not itself have to be a $\ph$-formula). 
\item A type $p$ is isolated iff it is $\aleph_0$-isolated. 
\end{enumerate}
 
\end{re}

\smallskip

The following Lemma (that follows from Fact \ref{fct:complete_char}) with play an important role in our constructions. It states that a locally isolated type over a complete set is \emph{always} weakly orthogonal to $P$.

\begin{lem}\label{lem:isolated_star}(Lemma 6.3 in \cite{Us24})
 Let $B$ be a complete set, $p \in S(B)$ locally isolated. Then $p \in S_*(B)$. 
\end{lem}

\begin{lem}\label{lem:li} (Lemma 7.1 in \cite{Us24}
Assume that $T$ is countable. 
\begin{enumerate}
\item 
Let $B$ be a stable set, $p(\bar x)$  be
a finite $m$-type over $B$, $\psi(\x,\y)$ a formula. Then there is
a finite $\psi$-type $q(\bar x)$ over $B$ 
such that 
$p(\bar x)\cup q(\bar x)$ is consistent, and it implies a complete $\psi$-type over $B$. 


\item 
Let $B$ be a stable set, $p(\bar x)$  be
a finite $m$-type over $B$. Then there is
$q(\bar x)$ such that $|q(\bar x)|\leq |T| = \aleph_0$,
$p(\bar x)\cup q(\bar x)$ consistent, and there is $r\in S_*(B)$ such
that $r$ is locally isolated, and $p(\bar x)\cup q(\bar x)\equiv r(\bar x)$.

In particular,  $r(\bar x)$ is
$\aleph_1$-isolated.

\end{enumerate}

\end{lem}


\begin{de}\label{dfn:lprimary}
Let $N$ be a model, $P^N \subseteq B \subseteq N$.
\begin{enumerate}
\item
	We say that the sequence $\d = \{d_i:i<\alpha\} \subseteq N$ is a \emph{local construction} over $B$ in $N$ if for all $i<\al$, the type $tp\{d_i/B\cup\{d_j:j<i\})$ is locally isolated.

\item
	We say that a set $C \subseteq N$ is $N$ is \emph{ locally constructible} (l.c.) over $B$ in $N$ if 
	there is a local construction $\d$ over $B$ in $N$. 
	
\end{enumerate}
\end{de}

\begin{de}\label{dfn:primary}
Let $N$ be a model, $P^N \subseteq B \subseteq N$.
\begin{enumerate}
\item
	We say that the sequence $\d = \{d_i:i<\alpha\} \subseteq N$ is a $\lam$-construction over $B$ in $N$ if for all $i<\al$, the type $tp\{d_i/B\cup\{d_j:j<i\})$ is $\lambda$-isolated.

\item
	We say that a set $C \subseteq N$ is $N$ is \emph{$\lam$-constructible} over $B$ in $N$ if 
	there is a $\lam$-construction $\d$ over $B$ in $N$. 
	
	In particular, we say that $N$ is
	\emph{$\lam$-constructible} over $B$ if 
there is a construction $N=B\cup\{d_i:i<\al\}$ such that for all $i<\lam$ the type 
$tp\{d_i/B\cup\{d_j:j<i\})$ is $\lambda$-isolated. 
\item 
	We say that a model $N$ is \emph{$\lambda$-primary} over $B$ if it is $\lam$-constructible and $\lam$-saturated. 
\end{enumerate}
\end{de}

\begin{de}\label{dfn:prime}
Let $N$ be a model, $P^N \subseteq B \subseteq N$.
\begin{enumerate}
\item 

	We say that $N$ is \emph{locally atomic} (l.a.) over $B$ if 
for every $\bar d\subseteq N$, $tp(\bar d,B)$ is
locally isolated over $B$. 

\item 

	We say that $N$ is \emph{$\lam$-atomic} over $B$ if 
for every $\bar d\subseteq N$, $tp(\bar d,B)$ is
$\lambda$-isolated over some $B_{\bar d}\subseteq B, |B_{\bar
d}|<\lambda$.
%
%
\end{enumerate}
\end{de}

\begin{obs}\label{obs:localimpliestplus}
	Let $A$ be a set, and $p\in S(A)$ locally isolated. Then $p$ is $|T|^+$-isolated. 
\end{obs}


\subsection{Locally constructible existence property}

\begin{de}
 Let $A$ be a complete set. We say that $A$ has the \emph{locally constructible (l.c.) existence property} if there exists $M \models T$ with $P^M \subseteq A \subseteq M$ such that $M$ is locally constructible over $A$. 
\end{de}

\begin{lem}
 \label{lem:fin-stable}
 Let $A$ be a stable set, $\a$ a finite tuple. Then the set $A\a$ is also stable. 
 \end{lem}
\begin{proof}
 Let $B = A\a$. By Fact \ref{fct:satstable}, it is enough to verify stability for $B' \equiv B$ saturated. Let $M$ be a model containing $B$. Expand the language by an additional unary predicates $Q$ for $A$ and $Q_1$ for $B$. Let $M'$ be saturated in the expanded language, and let $B' = Q_1^{M'}$, considered as a substructure of $M'$ in the original language. Clearly $B' \equiv B$, and is saturated. Similarly, $A'  = Q^{M'}$ is elementarily equivalent to $A$ (in the original language). Let $\a' = B'\setminus A'$, so it is still finite (of the same length as $\a$, which we will call $n$). Every type $\tp(\b/B') \in S^m_*(B')$ gives rise to a type of a tuple $\tp(\a\b/A') \in S^{n+m}_*(A')$, and two different types $\tp(\b_1/A')$, $\tp(\b_2/A')$ give rise to two different such types over $A'$. By stability of $A$, there are $\le |A'|^{|T|}$ types in $S^{n+m}_*(A')$; hence the same is the case for $S^m_*(B')$.
\end{proof}

\begin{lem}
 \label{lem:count-li}
 Assume that $T$ is countable, and let $A$ be a countable stable set. Then $A$ has the locally constructible existence property. 
\end{lem}
\begin{proof}
  	Let $\set{\ph_i(x,\a_i)\colon i<\om}$ list all formulae over $A$. Construct an increasing continuous sequence of sets $A_i$ such that:
	\begin{enumerate}	
	\item 
		$A_0 = A$
	\item
		$A_{i+1} = A_i \cup \set{b_i}$
	\item
		$\models \ph_i(b_i, \a_i)$
	\item
		$\tp(b_i/A_i) \in S_*(A_i)$ and is locally isolated
	\end{enumerate}
	This can be done by induction using Lemma \ref{lem:li} for the successor stages, since all $A_i$ are stable by Lemma \ref{lem:fin-stable}. Note that clause (iv) for $b_i, A_i$ implies by induction that all $A_j$ are complete, and $P^{A_j} = P^A$.  

\end{proof}

\section{Consequences of stability of models}

Again, in order to keep the paper relatively self-contained, in this section, we recall some of the main concepts and results from our previous article \cite{ShUs322a}, that were proven under the assumption that every model $M \models T$ is stable over $P$. Since our ultimate goal in this paper is to investigate consequences of $n$-stability for all $n<\om$ (and, as we shall see, stability of models is merely $1$-stability), this is a ``harmless'' assumption for the current investigation. Moreover, in \cite{Sh234}, the first author has shown that the negation of this assumption implies (consistent) non-structure. So even when it comes to approaching Conjecture \ref{con:gaifman1}, at least a ``consistent'', or, rather, ``absolute'' version thereof, this hypothesis is ``justified'' by the non-structure side of the global picture. 

So for the rest of the paper, we assume, in addition to Hypothesis \ref{hyp:1}: 

\begin{hyp}\label{asm:2} (\underline{Hypothesis 2}). 
	Every (equivalently, some) $M\prec \cC$ is stable over $P$ (as in Definition \ref{dfn:startypes}(ii)).
	
	In other words, $|S_*(M)| \le |M|^{|T|}$ for all $M\models T$.
\end{hyp}

All the results stated below have been proven assuming this Hypothesis.

\begin{theorem}\label{4.1}\label{15} (Theorem 7.2 in \cite{ShUs322a}) If $A$ is stable, $|A|\geq2$, then for every $\psi(\bar
x,\bar y)$ there is a quantifier free $\Psi _\psi (\bar y,\bar z)\in
L(T)$ such that whenever $B\equiv A$ and $ p\in S_*(B)$ then 
$p| \psi$ is
defined by $\Psi _\psi (\bar y,\bar d)$ for some $\bar d\subseteq B$,
i.e. $p| \psi=\{\psi(\bar x,\bar a):\bar a\in B, B\models
\Psi_\psi(\bar a,\bar d)\}$.
\end{theorem}

\begin{theorem}\label{17}\label{stable}(Theorem 7.3 in \cite{ShUs322a}) Let $A$ be complete and $\lambda =\lambda^{<\lambda}$.
The following are equivalent:
\begin{itemize}
\item[(i)]  $A$ is stable.
\item[(ii$)_\lambda$ ] If $A^\prime\equiv A$ is $\lambda$-saturated,
$\lambda =|A^\prime|>|T|$, then over $A^\prime$ there is a
$\lambda$-primary 
model
$M$.
\item[(iii$)_\lambda$] If $A^\prime\equiv A$ is $\lambda$-saturated,
$\lambda >|T|$, then every $m$-type $p$ over $A$, $|p|<\lambda $ can be
extended to a $\lambda$-isolated $q\in S_*(A^\prime)$.
\end{itemize}
\end{theorem}

\smallskip

We now recall some new concepts introduced in \cite{ShUs322a}, and their properties (see section 8 of \cite{ShUs322a}) for details). 

\begin{de}
$A\subseteq _t B$ if for every $\bar a\in A,\bar b\in B$ 
and $\psi\in L(T)$ such that $\models\psi(\bar b,\bar a)$ there is  
some $\bar
b^\prime\subseteq A$ such that $\models \psi(\bar b^\prime,\red{\bar a})$
\end{de}

\begin{de}

Suppose $A$ is stable, $p\in S_*(A)$ and $A\subseteq _t B$.
Then $q\in S(B)$ is a {\it stationarization} of $p$ over $B$ if 
for every $\psi\in L$ there is some definition $\Psi _\psi(\bar y,\bar a
_\psi)$ with $\bar a_\psi\subseteq A$ that defines both $p_\psi$ and $q _\psi$.
\end{de}

\begin{no}\label{no:ind}
\begin{enumerate}
\item
	We write $\a \ind_A B$ if $A$ is stable and $q=\tp(\a/B)$ is a stationarization of $p = \tp(a/A)$ (so in particular $p \in S_*(A)$ and $A \subseteq_t B$). In this case, will also write $q = p|B$.
\item
	We write $C \ind_A B$ if for every $\a \in C$  we have $\a \ind_A B$.
\end{enumerate}
\end{no}

\begin{lem}\label{station}\label{20} (\cite[8.5,8.6,8.7]{ShUs322a}) Assume $A$ is stable, $A\subseteq _t B$ and
$p\in S_*(A)$. Then:
\begin{itemize}
\item[(i)] $p$ has a stationarization $q$ \red{over $B$}.
\item[(ii)] It is unique: We can replace``some $\Psi _\psi(\bar y,\bar a
_\psi)$" by ``every...", so $q$ does not depend on its choice. 
\item[(iii)] If $B$ is complete, $q\in S_*(B)$.
\item[(iv)] If $A\subseteq _t B, A$ stable, $\bar c\subseteq\bar b,tp(\bar b/A)\in
S_*(A)$, then $tp(\bar c/A)\in S_*(A)$ and the stationarization of
$tp(\bar b/A)$ over $B$ includes the stationarization of $tp(\bar c/A)$
over $B$.
\end{itemize}
\end{lem}

\begin{co} 
	Let $q \in S_*(B)$  definable over $A \subseteq_t B$, $A$ a stable set. Then $q$ is the stationarization of $q\rest A$. 
\end{co}

\begin{lem}\label{mapextension}(Lemma 8.10 in \cite{ShUs322a})
	Let $A,B,C$ be sets such that $A\subseteq_t B$, $A\subseteq C$, 
	$C\ind_A B$ (see Notation \ref{no:ind}). Let $F$ be an elementary map from $B$ onto $B'$, $G$ be an elementary map from $C$ onto $C'$ such that $F\rest A = G\rest A$. Then $F \cup G$ is elementary.
\end{lem}

\begin{lem}\label{isolatedimplies} (Lemma 8.11 in \cite{ShUs322a})
 Let $A$ be $\lam$-saturated and stable, $A \subseteq_t B$, $N$ a $\lam$-saturated model $\lam$-atomic over $A$ such that 
$N \ind_A B$.
 
 Then $tp(N/A) \vdash tp(N/B)$; so the types $tp(N/A)$ and $tp(B/A)$ are weakly orthogonal.  
\end{lem}

\section{Good Systems}

Finally we are ready to formally define the main tool that is used in our proofs: good systems. 

\begin{no} 

\smallskip

\begin{itemize}

\item ${\cal P}(Y)=\{Z:Z\subseteq Y\}, {\cal P}^-(Y)={\cal
P}(Y)\setminus \{ Y\}$. 
\item We also use the standard set theoretic notation $n = \set{0, 1, \ldots, n-1}$. 
\end{itemize}

\end{no}

\begin{de}\label{system}\label{30}
\begin{itemize}
\item[(1.)] We say that $I$ is {\it weakly nice} if for some $n,
I\subseteq {\cal P}(n), {\cal P}(\{ 1,\dots ,n\})\subseteq I$ and $I$
is hereditary (i.e. $t\subseteq s$ and $s\in I$ implies $t\in I$).
We say $I$ is {\it nice} if in addition it is an initial segment of
${\cal P}(n)$ ordered lexicographically (identifying $s\in {\cal
P}(n)$ with its characteristic function; so $\{ t\in I:0\notin t\}$ is
an initial segment and for every $s\in I$ all subsets of $s$ precede
$s$ with respect to this ordering.) Let $n_I$ be this $n$. If $s$ is
the last element on this list, $0\in s$, then we say $(I,s)$ is {\it nice}.

\item[(2.)] If $I$ is weakly nice, we define by induction on $n_I$
that ${\cal S}=\langle A_s:s\in I\rangle $ is a {\it good system} (or {\it good
$I$-system}) if:
\begin{itemize}
\item[(i)] $I$ is nice,
\item[(ii)] $A_s\cap A_t=A_{s\cap t}$,
\item[(iii)] \begin{itemize}
             \item[(a)] $0\notin s\Rightarrow A_s\prec P^{\cal C}$,
             \item[(b)] $0\in s\Rightarrow A_s\prec {\cal C}$.
             \end{itemize}
\item[(iv)] $A_s\cap P^{\cal C}=A_{s\setminus\{ 0\}}$,
\item[(v)] if $n\geq 2$, then $\langle A_s: s\in {\cal P}(\{1,\dots
,n-2\})\rangle \prec \langle A_{s\cup\{ n-1\}}:s\in {\cal P}(\{ 1,\dots ,n-2\} )\rangle $
and these are good systems,
\item[(vi)] $\bigcup _{t\in {\cal P}^-(s)} A_t$ is stable when $0\in
s\in I$.
\item[(vii)] For any $\phi\in L$ and $\bar b_s\in A_s$ for $s\in I$
such that $\models \phi (\dots ,\bar b_s,\dots )$ there are $\bar
b^\prime _s, s\in I$, satisfying: $\bar b^\prime _s\in A_{s\setminus\{
n-1\}}$ and $\models \phi(\dots ,\bar b^\prime _s,\dots )$,
and, in addition,  if
$n-1\notin s$, then we can choose $\bar b^\prime_s=\bar b_s$.)
\item[(viii)] $\langle A_{s\cup\{ n-1\}}: s\in I_{[n-1]}\rangle $ 
and $\langle A_s:s\in
I^{[n-1]}\rangle $ are both good systems where
     
\begin{itemize}\item[(a)] $I^{[n-1]}=\{ s:n-1\notin s\in I\}$
               \item[(b)] $I_{[n-1]}=\{ s:s\cup\{ n-1\}\in I, (n-1)\notin s\}$.
\end{itemize}
\end{itemize}
\item[3.] We call ${\cal S}$ a {\it weakly good} (w.g.) system if
$(ii)-(iv)$ hold and a {\it medium good} (m.g.) system if $(v),(vi)$ hold
too. We call ${\cal S}$ a {\it presystem} if it has the right form.
A $(\lambda, I)$-system or $\lambda$-system means $|A_s|=\lambda$ for
all $s\in I$, and the system is called {\it stable} if $\bigcup _{s\in
I}A_s$ is a stable set.
\end{itemize}
\end{de}

\begin{lem}\label{31}
\begin{itemize}
\item[(i)] If $\langle A_s:s\in I\rangle \equiv \langle A^\prime _s:s\in I\rangle $ where these
presystems are treated as structures, then one is good if and only if
the other one is good; similarly for w.g., m.g.
\item[(ii)] If $\langle A_s:s\in I\rangle $ is good where $(I,t)$ is nice, then
$\langle A_s:s\in I\setminus\{ t\}\rangle $ is also good.
\item[(iii)] If $0=l_0<l_1<\dots <l_{m-1}<n_I$ and $J=\{ v\subseteq
m:\{ l_i:i\in v\}\in I\} , B_v=A_{\{l_i:i\in v\}}$ and $\langle A_s:s\in I\rangle $
is a good system, then so is $\langle B_v:v\in J\rangle $.
\end{itemize}
\end{lem}
\begin{proof}
(i): For clause $(vi)$ of
Definition \ref{system} recall that stability is  a property of the theory of a set. 
For the rest use
induction on $n_I$.

(ii): Straightforward checking.

(iii) Clauses $(i) - (v)$ are straightforward, and we prove clauses
$(vi) - (viii)$ by induction on $n$ (for all possible systems, $m$ and
$l_0<\dots <l_{m-1}$). If $l_{m-1}<n-1$ use the induction hypothesis
applied to $\langle A_s:\red{s}\in I^{[n-1]}\rangle$, 
which is a good system by Definition
\ref{system} $(viii)$.
For $l_{m-1}=n-1$ clauses $(vi) - (viii)$ for $\langle B_v:v\in J\rangle $ follow
from the corresponding clauses for $\langle A_s:S\in I\rangle $.
\end{proof}

\red{
\begin{remark}\label{32.5}
	Note that from clause (vii) of Definition \ref{30} it follows that if $\inseq{A}{s}{I}$ is a good system, then $\bigcup_{s \in I^{[n-1]}}A_s \subseteq_t \bigcup_{s \in I}A_s$. 
\end{remark}
}

\begin{lem}\label{1.4}\label{32}
\begin{itemize}
\item[(i)] Suppose $\langle A_v:v\in I\setminus\{ s\}\rangle $ is good, $(I,s)$ nice, then\\
$\bigcup_{v\in {\cal P}^-(s)} A_v\subseteq _t\bigcup _{v\in I,v\not=
s}A_v$.
\item[(ii)] Further, if $\langle A_v:v\in I\rangle $ is good, $(I,s)$ nice, then for
every $\bar b\in A_s,\\ tp(\bar b/\bigcup _{v\in I,v\not=
s}A_v)$ is definable over $\bigcup_{v\in {\cal P}^-(s)} A_v$. 

\red{Hence, combining with (i) and Lemma \ref{20}, $tp(\bar b/\bigcup _{v\in I,v\not=
s}A_v)$ is the stationarization of  $tp(\bar b/\bigcup_{v\in {\cal P}^-(s)} A_v)$. }
\end{itemize}
\end{lem}
\begin{proof} By induction on $n_I=n$ and $|I|$. Let $J=I\setminus \{
s\}$. If $n_I=0$ , the statements are vacuous. Let $n_I=1$,
necessarily $s=\{ 0\}$, ($0\in s$ as $(I,s)$ is nice), so
$I=\{\emptyset ,s\}$; ${\cal P}^-(s)=\{\emptyset\}=\{ v:v\in
I, v\not= s\}$, so $(i)$ holds trivially. As for $(ii)$, using the
same equality, the conclusion holds by Hypothesis \ref{hyp:1} \red{every type over $P^{\cC}$ is definable} as
$A_s\prec {\cal C}, P^{A_s}=A_\emptyset$ by Definition \ref{system}
$(iii),(iv)$.

So let $n_I\geq 2$. We first prove $(i)$. 
$\bigcup _{v\in J^{[n-1]}}A_v\subseteq _t\bigcup _{v\in J} A_v$ by
$(vii)$ of Definition \ref{system} \red{(see Remark \ref{32.5})}. 
By the induction hypothesis
applied to $\langle A_v:v\in J^{[n-1]}\rangle $ (which is a good system by $(viii)$
of the definition) we have $\bigcup _{v\in {\cal P}^-(s)} A_v\subseteq
_t\bigcup _{v\in J^{[n-1]}} A_v$ (as ${\cal P}^-(s)\subseteq J^{[n-1]}$). 
But $\subseteq _t$ is transitive,
so we get the desired conclusion of part $(i)$. 

So assume now $n-1\in s$. 
Let $t:=s\setminus\{ n-1\}$, so $0\in t$ and $t\subseteq t^\prime$
implies $t^\prime\in \{ t,s\}$. As $t$ is the second to last element
in $I$ with respect to lexicographical ordering, $(I\setminus\{
s,t\})$ is nice.

Suppose $\models\phi (\bar a,\bar b,\bar c)$ where $\bar a\in
A_t,\bar b\in \bigcup _{u\in {\cal P}^-(t)}A_{u\cup\{ n-1\}},\bar c\in
\bigcup_{v\in I,v\not= s}A_v$, and we should find $\bar c^*\in \cup
_{u\in {\cal P}^-(s)} A_u$ 
such that $\models \phi (\bar a,\bar b,\bar
c^*)$. Clearly without loss of generality $\bar c\in\bigcup_{v\in
I,v\not= s,t}A_v$. By the induction hypothesis on $(ii)$ applied to
$I\setminus\{ s\}$, $tp(\bar a/\bigcup_{v\in I,v\not= s,t}A_v)$ is
definable over $\bigcup _{v\in {\cal P}^-(t)} A_v$. So we can choose
$\Psi =\Psi_\phi (\bar y,\bar z,\bar d)$ with $\bar d\in \bigcup
_{v\in {\cal P}^-(t)} A_v$ which defines $tp_\phi(\bar a/\bigcup_{v\in
I,v\not= s,t}A_v)$. Hence $\models\Psi(\bar b,\bar c,\bar d)$.

Consider the system $\langle A_{u\cup\{ n-1\}}:u\in I_{[n-1]}
\setminus\{ t\}\rangle $.
(Note: since $I$ is an initial segment of ${\cal P}(n_I)$, $s$ last in
$I$ and $n-1\in s$, we have that $u\cup\{ n-1\}\in I$ whenever $u\in
I$.) This system is good by clause $(viii)$ of Definition \ref{system}
and the previous lemma \red{clause (ii)} and 
$\b\d \in \bigcup _{u \subset t} A_u$). So we can apply part $(i)$
of the induction hypothesis to find $\bar c^*\in \cup _{u\in {\cal
P}^-(t)}A_{u\cup\{ n-1\}}\red{\subseteq_t}$
\blue{$\bigcup_{u\in I_{n-1}, u\neq t} A_u$} so 
that $\models\Psi(\bar b,\bar c^*,\bar d)$, hence by the choice of $\bar
d$, $\models\phi(\bar a,\bar b,\bar c^*)$ as required.

Next we have to prove the induction step for $(ii)$.

%

\red{


First assume that $n-1\notin s$.

$\langle A_v:v\in I^{[n-1]}\rangle $ is a good
system, so by the induction hypothesis $tp(\bar a/\bigcup _{v\in
J^{[n-1]}}A_v)$ is definable over $\bigcup_{v\in {\cal P}^-(s)}A_v$ (note that $\a \in A_s$ and $s \in I^{[n-1]}$ by the assumption).

We claim that these same definitions work for $tp(\bar
a/\bigcup _{v\in J}A_v)$. Indeed, we know that for some $\d\in \bigcup_{v\in {\cal P}^-(s)}A_v$ for all $\b \in \bigcup_{t \in J^{[n-1]}} A_t$ we have $\ph(\a,\b) \iff \Psi_\ph(\b,\d)$. If for some $\b \in A_v \in bigcup _{v\in J}A_v$ the above is not the case, since $\bigcup _{v\in J^{[n-1]}}A_v\subseteq _t\bigcup _{v\in J}A_v$ (see Remark \ref{32.5}), we can find a ``counterexample'' in $\bigcup _{v\in J^{[n-1]}}A_v$,  a contradiction.

}

Now suppose $n-1\in s$. Let $\bar a\in A_s$ and $\phi(\bar x,\bar y)$ be a
formula. By Theorem \ref{4.1}
$tp_\phi(\bar a/\bigcup_{u\in {\cal P}^-(s)}A_u)$ is definable
by some $\psi(\bar y,\bar b)$ where $\bar b\in  \bigcup_{u\in {\cal
P}^-(s)} A_u$
since the domain is a stable set by part $(vi)$ of the definition.
Let $\bar b=\bar b_1\bar b_2$ where $\bar b_1\in A_t,\bar b_2 \in\bigcup
_{u\in {\cal P}^-(t)}A_{u\cup\{ n-1\}}$ where again $t:=s\setminus \{
n-1\}$.
Suppose $\psi(\bar y,\bar b)$ does not define $tp_\phi(\bar a/\bigcup
_{u\in I,u\not= s} A_u)$. Let $\bar c^*$ be a witness to this lack of
definability and $\bar c^*=\bar c^*_1\bar c^*_2,\bar c^*_1\in A_t,\bar
c^*_2\in\bigcup _{u\in I,u\not= s,t}A_{u\cup\{ n-1\}}$.

Hence we have $\models\phi(\bar a,\bar c^*_1,\bar c^*_2)$ if and only if
$\models\neg\psi (\bar c^*_1,\bar c^*_2,\bar b_1,\bar b_2)$.
By part $(viii)$ of the definition $\langle A_{u\cup\{ n-1\}}:u\in I_{[n-1]}\rangle $
is a good system, so by part $(vi)$ applied to the type $tp|{\phi _1}(\bar
a\bar c^*_1/\bigcup_{u\in {\cal P}^-(t)}A_{u\cup\{ n-1\}})$ (where $\phi
_1=\phi (\bar x,\bar y_1,\bar y_2)$) is defined by some $\theta(\bar
y_2,\bar d)$ for some $\bar d\in\bigcup _{u\in {\cal P}^-(t)}A_{u\cup\{
n-1\}}$ and by $(viii)$ of the definition and the induction hypothesis
this definition works for $\bigcup _{u\in I,u\not= s,t}A_u=\bigcup _{u\in
I_{[n-1]}, u\not= t}A_{u\cup\{ n-1\}}$.

As we can replace $\phi$ by $\neg\phi$, without loss of generality we
have $\models\phi (\bar a,\bar c^*_1,\bar c^*_2)\wedge \neg\psi (\bar
c^*_1,\bar c^*_2,\bar b_1,\bar b_2)$. So $\models\theta(\bar c^*_2,\bar
d)$. As we have already proved part $(i)$ for $(I,s)$, we can apply it to
find $\bar c^{**}_2\in\bigcup_{u\in {\cal P}^-(s)}A_u$ satisfying
$\models\theta(\bar c^{**}_2,\bar d)\wedge\neg\psi(\bar c^*_1\bar
c^{**}_2,\bar b_1,\bar b_2)$.
Now by the choice of $\theta(\bar y_2,\bar d)$ this says $\models \phi
(\bar a,\bar c^*_1,\bar c^{**}_2)$.
But $\bar c^*_1,\bar c^*_2\in \bigcup _{v\in {\cal P}^-(s)}A_v$, hence
by the choice of $\psi$ we have $\models\psi(\bar c^*_1,\bar c^*_2,\bar
b)$. But $\bar b=\bar b_1\bar b_2$, so the last statement contradicts
the second conjunct above, and we finish.
\end{proof}

%
%

We can finally conclude that the union of a good system is a complete set:

\begin{co}\label{33}
\begin{itemize}
\item[(i)] If $\langle A_s:s\in I\rangle $ is a good system, then $\bigcup _{s\in
I}A_s$ is complete.
\item[(ii)] If ${\cal S}^l=\langle A^l_t:t\in I\rangle $ are good systems for $l=1,2$,
$(I,s)$ is nice, $F$ is and elementary mapping from $\bigcup _{t\not=
s}A^1_t$ onto $\bigcup _{t\not= s}A_t^2$, $G$ is an elementary mapping
from $A^1_s$ onto $A_s^2$ and $F|\bigcup _{t\subset s}A^1_t=
G| \bigcup_{t\subset s}A^1_t$, then $G\cup F$ is an elementary
mapping.
\end{itemize}
\end{co}
\begin{proof}
By induction on $|I|$; without loss of generality $I\not= {\cal P}(\{
1,\dots , n_I-1\})$ \red{(that is, $I\supset {\cal P}(\{
1,\dots , n_I-1\})$)} since for $I={\cal P}(\{ 1,\dots n_I-1\})$ the claim
follows from condition $(iii)$ (b) of the definition. Let $\lambda
=\lambda ^{<\lambda}>|T|$ and let ${\cal S}^\prime =\langle A^\prime _t:t\in
I\rangle $ be such that:
\begin{itemize}
\item[(a)] $\langle A^\prime _t:t\in I\rangle \equiv \langle A_t:t\in I\rangle $
\item[(b)] ${\cal S}^\prime$ is $\lambda$-saturated
\item[(c)] ${\cal S}^\prime$ has cardinality $\lambda$
\end{itemize}

It suffices to prove the statement for ${\cal S}^\prime$. Let $s\in I$
be the last member of $I$. \red{By our ``wlog'' assumption, $0 \in s$, that is, $(I,s)$ is nice. } By condition $(vi)$ of the definition,
$\bigcup _{t\subset s}A^\prime_t$ is a stable set, so by Theorem
\ref{stable} we can find $M_s, \bigcup _{t\subset s}A^\prime
_t\subseteq M_s\subseteq A^\prime _s, M_s$ 
$\lambda$-primary
over $\bigcup _{t\subset s}A^\prime_t$. For $\bar
a\in M_s, tp(\bar a/\bigcup _{t\subset s}A^\prime_t$ is
$\lambda$-isolated. As $\bigcup_{t\in I, t\not= s}A^\prime_t$ is
complete (by the induction hypothesis) there is a $\lambda$-saturated
model $M,P^M\subseteq _{t\in I,t\not= s}A^\prime_t\subseteq M$. So for
$\bar a\in M_s$ we have $tp(\bar a/\bigcup _{t\subset
s}A^\prime_t)$ implies $tp(\bar a/\bigcup _{t\in I,t\not= s}A^\prime_t)$ \red{(see Lemma \ref{isolatedimplies})}, 
and by the choice of $M_s$ this type is realized in $M$. \gray{Similarly} \red{So}  
without loss of generality, $M_s\subseteq M$. Hence $\bigcup _{t\in
I,t\not= s}A_t^\prime\cup M_s$ is complete.

Now $A^\prime_{s\setminus\{ 0\}} =P^{A^\prime_s}\subseteq M_s\prec
A^\prime _s$, hence for every $\bar a\in A^\prime _s, tp(\bar a/M_s)$ is
stationary and definable. By the first part of the previous lemma 
(as for every $\bar a\in A^\prime _s, tp(\bar a/M_s) \in S_*(M_s)$ since
 $P^{A^\prime_s}\subseteq M_s$)
the same definition works for $tp(\bar a/M_s\cup\bigcup_{t\in I, t\not=
s}A^\prime_t)$. As $M_s\cup\bigcup_{t\in I, t\not=
s}A^\prime_t$ is complete, there is a $\lambda$-saturated $M^\prime,
A_{n_I\setminus\{ 0\}}=P^{M^\prime}\subseteq M_s\cup \bigcup_{t\in I, t\not=
s}A^\prime_t\subseteq M^\prime$. 
We can extend $tp(A^\prime_s/M_s)$ to a complete type over $M^\prime$
using the same definitions. So it necessarily extends
$tp(A^\prime_s/M_s\bigcup_{t\in I, t\not= s}A^\prime_t)$.
So without loss of generality $tp(A^\prime _s/M^\prime)$
is definable over $M_s$.

(ii):  \red{Follows from Lemma \ref{mapextension}}.
\end{proof}

\begin{de}
 We say that $T$ is \emph{$n$-stable} over $P$ if for every good system $\langle A_s:s\in {\cal P}^-(n)\rangle$, its union $B = \bigcup_{s \in {\cal P}^-(n)} A_s$ is stable over $P$. 
\end{de}

\begin{de}
 We say that $T$ has the $n$-existence property if for every good system $\langle A_s:s\in {\cal P}^-(n)\rangle$, there is $A_n$ such that  $\langle A_s:s\in {\cal P}(n)\rangle$ is a good system.
\end{de}

\section{Existence property}

\underline{Hypothesis:} Every good system $\langle A_s:s\in {\cal P}^-(n)\rangle$
is stable. That is, $T$ is $n$-stable over $P$ for all $n<\om$. 

\noindent \underline{Conclusion:} For every nice $I$ and good system
$\langle A_s:s\in I\rangle$, $\bigcup _{s\in I}A_s$ is stable.

The conclusion follows from the Hypothesis by clause (vi) of the definition of a good system. So in particular, Hypothesis \ref{asm:2} (every model is stable over $P$) follows from the Hypothesis of this section.

\begin{theorem}\label{small}\label{45}
If $\langle A_s:s\in I\rangle$ is a good system, $\lambda=\sum_{s\in I} |A_s|,
\lambda > |T|$, then we can find $A^s_\alpha$ for $s\in I,\alpha
<\lambda$ such that
\begin{itemize}
\item[(a)] $|A_\alpha^s| \le |\alpha |+|T|$. 
\item[(b)] $\langle A_\alpha ^s:\alpha <\lambda\rangle$ is increasing continuously.
\item[(c)] For each $\alpha$, $\langle A_\alpha ^s:s\in I\rangle\prec
\langle A_s:s\in
I\rangle$.
\item[(d)] $\langle A_\alpha ^s:s\in I\rangle$ is a good 
system.
\item[(e)] Let $J=I\cup\{ t\cup\{ n_I\}:t\in I\}$, and for $s\in J,
\alpha <\beta$ let
\[ {A^s}_{\{\alpha ,\beta\}}=\left\{ \begin{array}{ll}
                                   A^s_\alpha &\mbox{if $s\in I$}\\
                                   A^t_\beta  &\mbox {if \red{$s\in J$}, $n_\red{I}\in s,
                                                 t=s\setminus\{ n_I\}$}
                                     \end{array}
                             \right. \]

Then $\langle A^s_{\{\alpha ,\beta\}}:s\in J\rangle$ is a good system.

\end{itemize}
\end{theorem}

\begin{proof}
We can easily define (\blue{L\"{o}wenhem-Skolem})  the $A^s_\alpha$ such that $(a),(b)$ and $(c)$
hold. Now $(d)$ holds by \ref{31} $(i)$. Now, clause $(e)$ follows:

\underline{Claim}
If $\langle A^s_1:s\in I\rangle$ is a good system, $\langle A^s_0:s\in I\rangle\prec
\langle A_1^s:s\in I\rangle$, and $J$ and $A^s_{\{ 0,1\}}$ are defined
as in $(e)$, 
then $\langle A_{\{ 0,1\}}^s:s\in J\rangle$ is a good system.


\emph{Proof} (of claim): 
Let  $\langle s^l:l<2^{|I|}\rangle$ be the lexicographic enumeration of $J$. We
prove by induction on $m$, $|I\cup {\cal P}(\{ 1,\dots ,n_I\})|\leq
m\leq 2^{|I|}$ that $\langle A^{s^l}_{\{ 0,1\}}:l<m\rangle$ is a good system.

For $m=|I\cup {\cal P}(\{ 1,\dots ,n_I\})|$ easy.
For $m+1$ check Definition \ref{system}.

$(i)-(v)$ are easy. $(vi)$ holds as the induction hypothesis implies
that $\langle A_{\{ 0,1\}}^t:t\subset s<m\rangle$ is a good system, hence (\blue{by the hypothesis of the section}) $\bigcup
_{t\subset s_m}A^t_{\{ 0,1\}}$ is stable. Next, $(vii)$ holds as
$\langle A^s_0:s\in I\rangle\prec \langle A^s_1:s\in I\rangle$. Lastly $(viii)$ holds as
$\langle A^s_1:s\in I\rangle$ is a good system (by assumption) and $\langle A^s_0:s\in I\rangle$
is a good system (by \ref{31} $(i)$).
\end{proof}

\begin{theorem}\label{2.5} ($T$ countable) If $I={\cal P}^-(n), \langle A_s:s\in I\rangle$
is a good
system then over $\bigcup _{s\in I}A_s$ there is an 
locally constructible model. Moreover, if $tp(\bar c/\bigcup _sA_s)\in
S_*(\bigcup _sA_s)$, then we can find a 
locally constructible model over $\bigcup _sA_s\cup \bar c$.
\end{theorem}


We have the following easy fact:
\begin{fact}\label{47}
If $B$ is 
locally constructible over $A$, $A\subseteq _tC$,
then $B$ is 
locally constructible over $C$ by the same
sequence, $A\cup B\subseteq _t C\cup B$, and $\tp(B/A) \vdash \tp(B/C)$.
\end{fact}

\begin{proof} (of Theorem \ref{2.5}) We prove the theorem by induction
on $\lambda=\sum |A_s|$. Remember $I={\cal P}^-(n)$.

For $\lambda\leq\aleph _0$, we inductively define an $\omega$-sequence
$\langle a_i:i<\omega\rangle$ such that $tp(\bar c^\frown\langle a_0,\dots
,a_i\rangle /\cup A_s)$ is
locally isolated.

For $\lambda > \aleph _0$ choose $A^s_\alpha, \alpha <\lambda$ as in
Theorem \ref{45} such that $tp(\bar c/\bigcup _sA_s)$ is definable over
$\bigcup A_0^s$. We define $A^{\{0,\dots ,n-1\}}_\alpha$ by induction
on $\alpha$ such that it is increasing continuously, $A^{\{0,\dots
,n-1\}}_0$ is 
locally constructible over $\bigcup
\{A^s_0:s\in I\}$ and for $\alpha <\lambda, A^{\{0,\dots
,n-1\}}_{\alpha +1}$ is 
locally constructible over
$A^{\{0,\dots ,n-1\}}_\alpha\cup\{ A^s_{\alpha +1}:s\in I\}$. \blue{$|T|$ is countable, so use the case $\lam = \aleph_0$}

Since (Fact \ref{47})

$$ tp_*(A^{\{0,\dots ,n-1\}}_{\alpha +1}/\bigcup _s A^s_{\alpha
+1}\cup A^{\{0,\dots ,n-1\}}_\alpha)\vdash tp_*((A^{\{0,\dots
,n-1\}}_{\alpha +1}/\bigcup _s A_s)$$

and

$$tp_*(A^{\{0,\dots ,n-1\}}_0/\bigcup _s A^s_0\cup\bar c)\vdash
tp_*(A^{\{0,\dots ,n-1\}}_0/\bigcup _s A_s\cup\bar c)$$

we have

$$tp_*(A^{\{0,\dots ,n-1\}}_\alpha /\bigcup _s A^s_\alpha \cup\bar
c)\vdash
tp_*(A^{\{0,\dots ,n-1\}}_\alpha /\bigcup _s A_s\cup\bar c)$$.

Let for $s\in {\cal P}^-(n+1)$,

\[ B^s_\alpha =\left\{ \begin{array}{ll}
  A^s_\alpha &\mbox{if $n\notin s,s\neq\{ 0,\ldots ,n-1\}$}\\
  A^{s\setminus\{ n\}}_{\alpha +1} &\mbox {if $n\in s$}
                         \end{array}
                         \right. \]

Note that in the notation  of Theorem \ref{small} for $I={\cal
P}^-(n)$, $B^s_\alpha= B^s_{\{\alpha ,\alpha +1\}}$.
So $\langle B^s_\alpha:s\in {\cal P}^-(n+1)\rangle$ is a good
system.

We use the induction hypothesis on $\lambda$ to carry the induction
step from $\alpha$ to $\alpha +1$.
\end{proof}

\begin{co} (T countable) Assume that $T$ is $n$-stable over $P$ for all $n < \om$. Then $T$ has $n$-existence for all $n<\om$. Moreover: let  $\langle A_s:s\in {\cal P}^-(n)\rangle$ be a good $\cP^-(n)$-system. Then its union $B = \bigcup_{s \in  \cP^-(n)} A_s$ has the locally constructible existence property, 
\end{co}

\begin{co}
 (T countable). Assume that $T$ is $n$-stable over $P$  for all $n < \om$. Then $T$ has the Gaifman property. Moreover, every $N\models T^P$ has the locally constructible existence property.  
\end{co}

The last Corollary is just a particular case of the previous one with $n=1$.

\medskip

In light of the results above, we restate our generalized version of Gaifman's Conjecture:

\begin{con}
 (T countable) Assume that $T$ is $n$-unstable over $P$ for some $n$. Then for every regular cardinal $\lam$ big enough, and every $\mu \ge \lam$, $T$ has $2^\lam$ models of cardinality $\mu$, which are non-isomorphic over $P$.  
\end{con}

Proving the last conjecture as stated may be challenging. We therefore state a weaker version of it, which, we believe, may be more attainable:

\begin{con}
 (T countable) Assume that $T$ is $m$-stable over $P$ for all $m<n$, but $n$-unstable for some $n$. Then for every regular cardinal $\lam$ big enough, and every $\mu \ge \ka = \lam^{+n}$, $T$ has $2^\ka$ models of cardinality $\mu$, which are non-isomorphic over $P$.  
\end{con}

An even more ``reasonable'' goal may be proving that for every regular $\lambda$ (perhaps $\lambda = \lambda^{<\lambda}$), there is a forcing extension of the universe that does not collapse any cardinals, in which we have non-structure as stated in the last Conjecture. We plan to address these questions in subsequent work.



\bibliography{common.bib}

\def\cprime{$'$}
\begin{thebibliography}{BPW23}

\bibitem[Afs14]{MR3343525}
Bijan Afshordel.
\newblock Generic automorphisms with prescribed fixed fields.
\newblock {\em J. Symb. Log.}, 79(4):985--1000, 2014.

\bibitem[BPW23]{benedikt-pred}
Michael Benedikt, C\'{e}cilia Pradic, and Christoph Wernhard.
\newblock Synthesizing nested relational queries from implicit specifications.
\newblock In {\em Proceedings of the 42nd ACM SIGMOD-SIGACT-SIGAI Symposium on
  Principles of Database Systems}, PODS '23, page 33–45, New York, NY, USA,
  2023. Association for Computing Machinery.

\bibitem[Gai]{Ga}
Haim Gaifman.
\newblock Characterisations of uniqueness and rigidity properties.
\newblock {\em unpublished note}.

\bibitem[Gai74]{gaifman}
Haim Gaifman.
\newblock Operations on relational structures, functors and classes. {I}.
\newblock In {\em Proceedings of the {T}arski {S}ymposium ({P}roc. {S}ympos.
  {P}ure {M}ath., {V}ol. {XXV}, {U}niv. {C}alifornia, {B}erkeley, {C}alif.,
  1971)}, Proc. Sympos. Pure Math., Vol. XXV, pages 21--39. Published for the
  Association for Symbolic Logic by the American Mathematical Society,
  Providence, RI, 1974.

\bibitem[GK10]{Goodrick_Kolesnikov_2010}
John Goodrick and Alexei Kolesnikov.
\newblock Groupoids, covers, and 3-uniqueness in stable theories.
\newblock {\em The Journal of Symbolic Logic}, 75(3):905–929, 2010.

\bibitem[GKK11]{goodrick2011amalgamationfunctorshomologygroups}
John Goodrick, Byunghan Kim, and Alexei Kolesnikov.
\newblock Amalgamation functors and homology groups in model theory, 2011.

\bibitem[GKK13]{Goodrick_Kim_Kolesnikov_2013}
John Goodrick, Byunghan Kim, and Alexei Kolesnikov.
\newblock Homology groups of types in model theory and the computation of
  h2(p).
\newblock {\em The Journal of Symbolic Logic}, 78(4):1086–1114, 2013.

\bibitem[Hod99]{Hod-cat1}
Wilfrid Hodges.
\newblock Relative categoricity in abelian groups.
\newblock In {\em Models and computability ({L}eeds, 1997)}, volume 259 of {\em
  London Math. Soc. Lecture Note Ser.}, pages 157--168. Cambridge Univ. Press,
  Cambridge, 1999.

\bibitem[Hod02]{Hod-cat3}
Wilfrid Hodges.
\newblock Relative categoricity in linear orderings.
\newblock In {\em Logic and algebra}, volume 302 of {\em Contemp. Math.}, pages
  235--248. Amer. Math. Soc., Providence, RI, 2002.

\bibitem[Hru24]{hrushovski2024groupoidsimaginariesinternalcovers}
Ehud Hrushovski.
\newblock Groupoids, imaginaries and internal covers.
\newblock {\em arXiv: Logic}, 2024.

\bibitem[HY09]{Hod-cat2}
Wilfrid Hodges and Anatoly Yakovlev.
\newblock Relative categoricity in abelian groups. {II}.
\newblock {\em Ann. Pure Appl. Logic}, 158(3):203--231, 2009.

\bibitem[Kir13]{Kirby2013}
Jonathan Kirby.
\newblock Finitely presented exponential fields.
\newblock {\em Algebra and Number Theory}, 7(4):943--980, 2013.

\bibitem[KZ14]{KiZil}
Jonathan Kirby and Boris Zilber.
\newblock Exponentially closed fields and the conjecture on intersections with
  tori.
\newblock {\em Ann. Pure Appl. Logic}, 165(11):1680--1706, 2014.

\bibitem[Mor65]{Mor}
Michael Morley.
\newblock Categoricity in power.
\newblock {\em Trans. Amer. Math. Soc.}, 114:514--538, 1965.

\bibitem[PS85]{PiSh130}
Anand Pillay and Saharon Shelah.
\newblock Classification theory over a predicate, {I}.
\newblock {\em Notre Dame J. Formal Logic}, 26(4):361--376, 1985.

\bibitem[She86]{Sh234}
Saharon Shelah.
\newblock Classification over a predicate, {II}.
\newblock In {\em Around classification theory of models}, volume 1182 of {\em
  Lecture Notes in Math.}, pages 47--90. Springer, Berlin, 1986.

\bibitem[She90]{Sh:c}
S.~Shelah.
\newblock {\em Classification theory and the number of nonisomorphic models},
  volume~92 of {\em Studies in Logic and the Foundations of Mathematics}.
\newblock North-Holland Publishing Co., Amsterdam, second edition, 1990.

\bibitem[SU22]{ShUs322a}
Saharon Shelah and Alexander Usvyatsov.
\newblock Classification over a predicate - the general case, part 1 -
  structure theory.
\newblock {\em arXiv: Logic}, 2022.

\bibitem[Usv24]{Us24}
Alexander Usvyatsov.
\newblock On the existence property over a predicate.
\newblock {\em arXiv: Logic}, 2024.

\end{thebibliography}
\bibliographystyle{alpha}

\Addresses

\end{document}